\documentclass[10pt]{elsart}               

\usepackage{amsmath,amssymb,enumerate,subfigure,multirow,hhline,color}
\usepackage{xcolor}
\usepackage{graphicx}
\usepackage{graphics}

\usepackage{epsfig,psfrag}
\usepackage{tabularx}
\usepackage{tabulary}
\usepackage{algorithm}
\usepackage{algpseudocode}

\usepackage[sort]{cite}

\usepackage{hyperref}
\hypersetup{breaklinks=true,colorlinks=false,linkcolor=blue,citecolor=blue}

\usepackage[noabbrev,capitalise,nameinlink]{cleveref}

\newtheorem{theorem}{Theorem}
\newtheorem{example}{Example}
\newtheorem{assumption}{Assumption}
\newtheorem{remark}{Remark}

\newtheorem{problem}{Problem}
\newtheorem{lemma}{Lemma}
\newtheorem{definition}{Definition}

\newenvironment{proof}{\begin{pf}}{\qed\end{pf}}

\allowdisplaybreaks

\begin{document}

\begin{frontmatter}

\title{Lossless Convexification and Duality \thanksref{footnoteinfo}\thanksref{footnoteinfo2}}

\thanks[footnoteinfo]{This paper was not presented at any IFAC meeting.}
\thanks[footnoteinfo2]{The work was supported by the BK21 FOUR from the Ministry of
Education (Republic of Korea) (Corresponding author: Donghwan Lee).
This work was supported by Institute of Information communications
Technology Planning Evaluation (IITP) grant funded by the Korea government
(MSIT)(No.2022-0-00469)}

\author[KAIST]{Donghwan Lee}\ead{donghwan@kaist.ac.kr},

\address[KAIST]{Department of Electrical Engineering,
KAIST, Daejeon, 34141, South Korea}


\begin{abstract}                          
The main goal of this paper is to investigate the strong duality of non-convex semidefinite programming problems (SDPs). In the optimization community, it is well-known that a convex optimization problem satisfies strong duality if Slater's condition holds. However, this result cannot be directly generalized to non-convex problems. In this paper, we prove that a class of non-convex SDPs with special structures satisfies strong duality under Slater's condition. Such a class of SDPs arises in SDP-based control analysis and design approaches. Throughout the paper, several examples are given to support the proposed results. We expect that the proposed analysis can potentially deepen our understanding of non-convex SDPs arising in the control community and promote their analysis based on KKT conditions.
\end{abstract}

\end{frontmatter}

\section{Introduction}

The field of computational control analysis and design has witnessed substantial advancements due to the development of convex optimization~\cite{Boyd2004} and semidefinite programming (SDP)~\cite{vandenberghe1996semidefinite} techniques, as investigated in various studies~\cite{Boyd1994,de1999new,geromel2007h,el2000advances}. In the control community, it is widely known that several important control design problems, such as the state-feedback stabilization and linear quadratic regulator (LQR) problems, can be effectively represented as convex SDP problems, which allow for efficient solution through convex optimization techniques. Interestingly, these convex SDP representations often originate as non-convex SDPs, which however can be efficiently transformed into convex forms through a change of variables~\cite{geromel1998static}. For example, the static state-feedback stabilization problem for a discrete-time linear system is initially posed as minimizing a nonconvex objective function that includes the spectral radius of the closed-loop system matrix. However, the problem can be efficiently solved through equivalent convex SDP problems~\cite{geromel1998static}. For convenience, let us call them losslessly convexifiable problems.

Theoretical studies of the losslessly convexifiable problems arising in control community hold the potential to provide valuable insights into the original problems, which will not only deepen our understanding but also facilitate the development of new and more efficient solution approaches for a range of optimization problems. However, these questions have not been fully explored to date. One of such unanswered questions is related to the strong duality property~\cite{Boyd2004}. Strong duality has been often studied for convex optimization problems in control field, e.g.,~\cite{vandenberghe1996semidefinite,Boyd1994,de2002extended,yao2001stochastic,rami2000linear}, because the connections between the strong duality property and the associated problems sometimes provide fruitful theoretical avenues and additional insights on the classical control theory such as the Lyapunov theory~\cite{henrion2001rank}, LQR problem~\cite{yao2001stochastic,rami2000linear,lee2019primal,gattami2010generalized}, Kalman-Yakubovich-Popov (KYP) lemma~\cite{balakrishnan2003semidefinite,you2013lagrangian,you2015primal}. However, it has not been fully investigated for the aforementioned losslessly convexifiable problems until now.

This paper proves that the losslessly convexifiable problems indeed satisfy strong duality, which plays a pivotal role in establishing the theoretical foundation for using the primal-dual algorithm as a reasonable approach to finding solutions to the original constrained optimization problems. These primal-dual algorithms are designed to solve the min-max problem formulations~\cite{bertsekas1999nonlinear} of the original constrained optimization problems, and numerous studies~\cite{lee2019primal, farjadnasab2022model, clarke2022low, li2022model, esmzad2023maximum} have recently developed primal-dual algorithms for various control design problems, yielding effective solutions. Therefore, the results presented in this paper can play a vital role for the theoretical analysis of the primal-dual methods in~\cite{lee2019primal, farjadnasab2022model, clarke2022low, li2022model, esmzad2023maximum}. Moreover, the presented results can offer a quite general and simple framework for establishing strong duality in various optimization problems in engineering.

Lastly, it is important to note that losslessly convexifiable problems have been explored in recent studies such as those by~\cite{sun2021learning,mohammadi2019global}. These studies have established an interesting link between this class of problems and the so-called ``gradient dominance condition,'' which is pivotal in gradient-based policy search algorithms~\cite{fazel2018global,bu2019lqr,duan2023optimization,duan2022optimization,duan2023optimization}. These contributions significantly enhance our understanding by showing that in policy optimization problems within control and reinforcement learning, the gradient dominance property of the underlying objective function remains valid despite its non-convex nature, provided the problems are losslessly convexifiable. This property is vital as it allows for the application of gradient-based policy search algorithms~\cite{fazel2018global,bu2019lqr,duan2023optimization,duan2022optimization,duan2023optimization} to efficiently find globally optimal solutions. The potential connection between the findings of these papers and our work presents a valuable avenue for future exploration.

{\bf Notation}: The adopted notation is as follows: ${\mathbb R}$: set of real
numbers\; ${\mathbb R}^n$: $n$-dimensional Euclidean space; ${\mathbb R}^{n \times m}$: set of all $n\times m$ real matrices; $A^T$: transpose of matrix $A$; $A\succ 0$ ($A\prec 0$, $A\succeq 0$, and $A\preceq 0$, respectively): symmetric positive definite (negative definite, positive semi-definite,
and negative semi-definite, respectively) matrix $A$; $A\succ B$ ($A\prec B$, $A\succeq B$, and $A\preceq B$, respectively): $A-B$ is symmetric positive definite (negative definite, positive semi-definite,
and negative semi-definite, respectively); $I$: identity matrix with appropriate dimensions; ${\mathbb S}^n$: symmetric $n\times n$ matrices; ${\mathbb S}^n_+$: cone of symmetric $n\times n$ positive semi-definite matrices; ${\mathbb S}^n_{++}$: symmetric $n \times n$ positive definite matrices; $Tr(\cdot)$: trace of matrix $(\cdot)$; $\rho(A)$: spectral radius of a square matrix $A$, where the spectral radius stands for the maximum of the absolute values of its eigenvalues; $*$ inside a matrix: transpose of its symmetric term; s.t.: subject to; ${\bf relint}({\mathcal D})$: relative interior of a set ${\mathcal D}$; ${\bf dom}(f)$: domain of a function $f$.

\section{Problem Formulation and Preliminaries}

In this section, we briefly summarize basic concepts of the standard Lagrangian duality theory in~\cite{Boyd2004}. Let us consider the following optimization problem with matrix inequalities (semidefinite programming, SDP), which is our main concern in this paper.
\begin{problem}[Primal problem]\label{prob:1}
Solve for $x \in {\mathbb R}^n$
\begin{align*}
p^*:=&\min_{x \in {\mathbb R}^n} f(x)\quad {\rm{s.t.}}\quad \Phi_i(x) \preceq 0,\quad i\in \{1,2,\ldots,N \}
\end{align*}
where $x\in {\mathbb R}^n$, $\Phi_i: {\mathbb R}^n \to {\mathbb S}^m$ is a continuous matrix function for all $i\in \{1,2,\ldots,N \}$, ${\hat n}$ is a positive integer, and $f: {\mathbb R}^n \to {\mathbb R}$ is a continuous objective function.
\end{problem}
Note that in~\cref{prob:1}, we assume that the minimum point exists. Moreover,
we assume that the domain, denoted by ${\mathcal D}: = {\bf{dom}}\,(f)$, is nonempty. An important property of~\cref{prob:1} that arises frequently is convexity.
\begin{definition}[Convexity]
\cref{prob:1} is said to be convex if $f$ is a convex function, and the feasible set, ${\mathcal F}:=\{ x \in {\mathbb R}^n :\Phi_i(x) \preceq 0,i\in \{1,2,\ldots,N \}\}$, is convex.
\end{definition}
Note that for the feasible set, ${\mathcal F}:= \{ x \in {\mathbb R}^n :\Phi_i(x) \preceq 0,i\in \{1,2,\ldots,N \}\}$, to be convex, $\Phi_i(x)$ needs to be linear or convex in $x$ for all $i\in \{1,2,\ldots,N \}$~\cite{Boyd2004}. Another essential concept is the relative interior~\cite[pp.~23]{Boyd2004} defined below.
\begin{definition}[Relative interior]\label{def:relative-interior}
The relative interior of the set $\mathcal D$ is defined as
\[
{\bf{relint}}({\mathcal D}): = \{ x \in {\mathcal D}:B(x,r) \cap {\bf{aff}}({\mathcal D}) \subseteq {\mathcal D}\,\,{\rm{for}}\,\,{\rm{some}}\,\,r > 0\},
\]
where $B(x,r)$ is a ball with radius $r>0$ centered at $x$, and ${\bf{aff}}({\mathcal D})$ is the affine hull of ${\mathcal D}$ defined as the set of all affine combinations of points in the set $\mathcal D$~\cite[pp.~23]{Boyd2004}.
\end{definition}

Associated with~\cref{prob:1}, the Lagrangian function~\cite{Boyd2004} is defined as
\[
L(x,\bar \Lambda ): = f(x) + \sum\limits_{i = 1}^N {Tr(\Lambda _i \Phi _i (x))}
\]
for any $\Lambda_i \in {\mathbb S}^m_+,i\in \{1,2,\ldots,N \}$, called the Lagrangian multiplier, where $\bar \Lambda : = (\Lambda _1 ,\Lambda _2 , \ldots ,\Lambda _N )$. For any $\Lambda_i \in {\mathbb S}^m_+,i\in \{1,2,\ldots,N \}$, we define the dual function as
\[
g(\bar \Lambda ): = \mathop {\inf }\limits_{x \in {\mathbb R}^n } L(x,\bar \Lambda ) = \mathop {\inf }\limits_{x \in {\mathbb R}^n } \left( {f(x) + \sum\limits_{i = 1}^N {Tr(\Lambda _i \Phi _i (x))} } \right).
\]
It is known that the dual function yields lower bounds on the optimal value $p^*$:
\begin{align}
g(\bar \Lambda ) \le p^*\label{eq:13}
\end{align}
for any Lagrange multiplier, $\Lambda_i \in {\mathbb S}^m_+, i\in \{1,2,\ldots,N \}$. The Lagrange dual problem associated with~\cref{prob:1} is defined as follows.
\begin{problem}[Dual problem]\label{prob:2}
Solve for $\Lambda_i \in {\mathbb S}^m_+, \forall i\in \{1,2,\ldots,N \}$
\begin{align*}
&d^*:=\sup_{\Lambda_i \in {\mathbb S}^m_+, \forall i\in \{1,2,\ldots,N \}} g (\bar \Lambda).
\end{align*}
\end{problem}
The dual problem is known to be concave even if the primal is not. In this context, the original~\cref{prob:1} is sometimes called the primal problem. Similarly, $d^*$ is called the dual optimal value, while $p^*$ is called the primal optimal value. The inequality~\eqref{eq:13} implies the important inequality
\[
d^*  \le p^*,
\]
which holds even if the original problem is not convex. This property is called weak duality, and the difference, $p^* - d^*$ is called the optimal duality gap. If the equality $d^* = p^*$ holds, i.e., the optimal duality gap is zero, then we say that strong duality holds.
\begin{definition}[Strong duality]
If the equality, $d^* = p^*$, holds, then we say that strong duality holds for~\cref{prob:1}.
\end{definition}
There are many results that establish conditions on the problem under
which strong duality holds. These conditions are called constraint qualifications. Once such constraint qualification is Slater's condition, which is stated below.
\begin{lemma}[Slater's condition]
Suppose that~\cref{prob:1} is convex. If there exists an $x\in  {\bf relint}({\mathcal D})$ such that
\[
\Phi_i (x) \prec 0,\quad i\in \{1,2,\ldots,N \}
\]
then strong duality holds, where ${\bf relint}({\mathcal D})$ is the relative interior~\cite[pp.~23]{Boyd2004} defined in~\cref{def:relative-interior}.
\end{lemma}

Without the constraint qualifications, such as the Slater's condition, strong duality does not hold in general.
A natural question is, under which conditions the strong duality holds for non-convex problems?
Based on the ideas of Slater's condition, we will explore a class of non-convex problems which satisfies strong duality throughout the paper. For more comprehensive discussions on the duality, the reader is referred to the monograph~\cite{Boyd2004}.
\begin{assumption}\label{assumption:slater}
Slater's condition holds for~\cref{prob:1}.
\end{assumption}

\section{Main results}

\subsection{Lossless convexification}

In this subsection, we will study convexification of matrix inequality constrained optimizations, which have a special property to be addressed soon. Toward this goal, let us consider the following optimization problem.
\begin{problem}\label{prob:4}
Solve
\begin{align*}
&\inf_{v \in h({\mathcal D})} f'(v)\quad {\rm{s.t.}}\quad \Phi_i' (v) \preceq 0,\quad \forall i\in \{1,2,\ldots,N \}
\end{align*}
for some mapping $h$ such that $h({\mathcal D})$ is convex, where $\Phi_i': {\mathbb R}^n \to {\mathbb S}^{m'},i\in \{1,2,\ldots,N \}$ and $f': {\mathbb R}^n \to {\mathbb R}$ are convex, and $f$ and $\Phi$ can be expressed as
\begin{align*}
f (x) =& f' (h(x)) = (f'  \circ h)(x)\\
\Phi_i(x) =& \Phi_i'(h(x)) = (\Phi_i' \circ h)(x),\quad \forall i\in \{1,2,\ldots,N \}.
\end{align*}
\end{problem}

Note that~\cref{prob:4} is convex, and hence will be called a convexification of~\cref{prob:1}.
In particular, we will consider a special convexification called the lossless convexification defined below.
\begin{definition}[Lossless convexification]\label{def:lossless-convexification}
Consider the following feasible sets associated with~\cref{prob:1} and~\cref{prob:4}:
\begin{align}
{\mathcal F}: =& \{ x\in {\mathcal D} :,\Phi_i (x) \preceq 0,\forall i\in \{1,2,\ldots,N \}\},\label{eq:24}\\
{\mathcal F}':=& \left\{ {v \in h({\mathcal D}): \Phi_i'(v) \preceq 0,\forall i\in \{1,2,\ldots,N \}} \right\},\label{eq:25}
\end{align}
respectively, and suppose that the mapping $h$ whose domain restricted to $\mathcal F$ and codomain restricted to ${\mathcal F}'$, i.e., $h:{\mathcal F}  \to {\mathcal F}^\prime$, is a surjection. Then,~\cref{prob:4} is said to be a lossless convexification of~\cref{prob:1}.
\end{definition}

\begin{remark}
Note that $\mathcal F$ may be non-convex, while ${\mathcal F}'$ is convex. Moreover, it is important to note that the surjective mapping $h$, which is essential for verifying lossless convexification, can be identified manually and algebraically for each specific problem, yet there exists no universal method capable of automatically and algorithmically determining such a mapping.
\end{remark}
\begin{remark}
The studies in~\cite{sun2021learning,mohammadi2019global} introduce a novel concept akin to ``lossless convexification,'' establishing a crucial link between this concept and the so-called ``gradient dominance condition'' pivotal in gradient-based policy search algorithms~\cite{fazel2018global,bu2019lqr,duan2023optimization,duan2022optimization,duan2023optimization}. These papers make a significant contribution to our understanding by demonstrating that, in policy optimization problems within control and reinforcement learning, the presence of lossless convexification ensures the gradient dominance property of the underlying objective function, despite its non-convex nature. This property is crucial as it enables the application of gradient-based policy search algorithms~\cite{fazel2018global,bu2019lqr,duan2023optimization,duan2022optimization,duan2023optimization} to find globally optimal solutions. It is clear that there is a strong connection between the findings of these papers and our work. However, a more thorough analysis of this relationship would require extensive research beyond the scope of this paper, making it an intriguing topic for future exploration.
\end{remark}

In the sequel, simple examples are given to clearly illustrate the main notions of the lossless convexification.
\begin{example}\label{ex:1}
Let us consider the optimization
\begin{align}
&\min _{x \in {\mathbb R}} \quad f(x) = x^2\label{eq:3}\\
&{\rm{s}}{\rm{.t}}.\quad \Phi _1 (x) =  - x^2  + 1 \le 0,\nonumber
\end{align}
where $x \in {\mathbb R}$ is a decision variable. The corresponding constrained optimal solution is $x^* = 1$, the primal optimal value is $p^* = 1$, and ${\mathcal D} = {\mathbb R}$.
The feasible set is
\begin{align*}
{\mathcal F} =& \{ x \in {\mathbb R}: - x^2  + 1 \le 0\}\\
=& \{ x \in {\mathbb R} : x \le -1\}\cup \{ x \in {\mathbb R}: x \geq 1\},
\end{align*}
which is non-convex as depicted in~\cref{fig:0}.
Therefore, the optimization in~\eqref{eq:3} is non-convex. Now, let us consider the mapping
\[
x \mapsto h(x) =  - x^2,
\]
and the corresponding change of variable
\[
v = h(x) =  - x^2.
\]

Then, a convexification using $h$ is
\begin{align}
&\min _{v \in {\mathbb R}} \quad f'(v) =  - v\label{eq:5}\\
&{\rm{s}}{\rm{.t}}.\quad \Phi _1 '(v) = v + 1 \le 0.\nonumber
\end{align}
For this convexified problem, the corresponding feasible set is
\[
{\mathcal F}' = \{ v \in {\mathbb R}:v + 1 \le 0\},
\]
which is convex. Moreover, $h: {\mathcal F} \to {\mathcal F}'$ is a surjection because for all $v \in {\mathcal F}'$, we can find a function  $q(v) = \sqrt{-v}$ so that $q(v) \in {\mathcal F}$. Therefore,~\eqref{eq:3} is a lossless convexification of~\eqref{eq:5} by~\cref{def:lossless-convexification}. The overall idea is summarized in~\cref{fig:0}.
\begin{figure}[h!]
\centering\includegraphics[width=8.5cm,height=5cm]{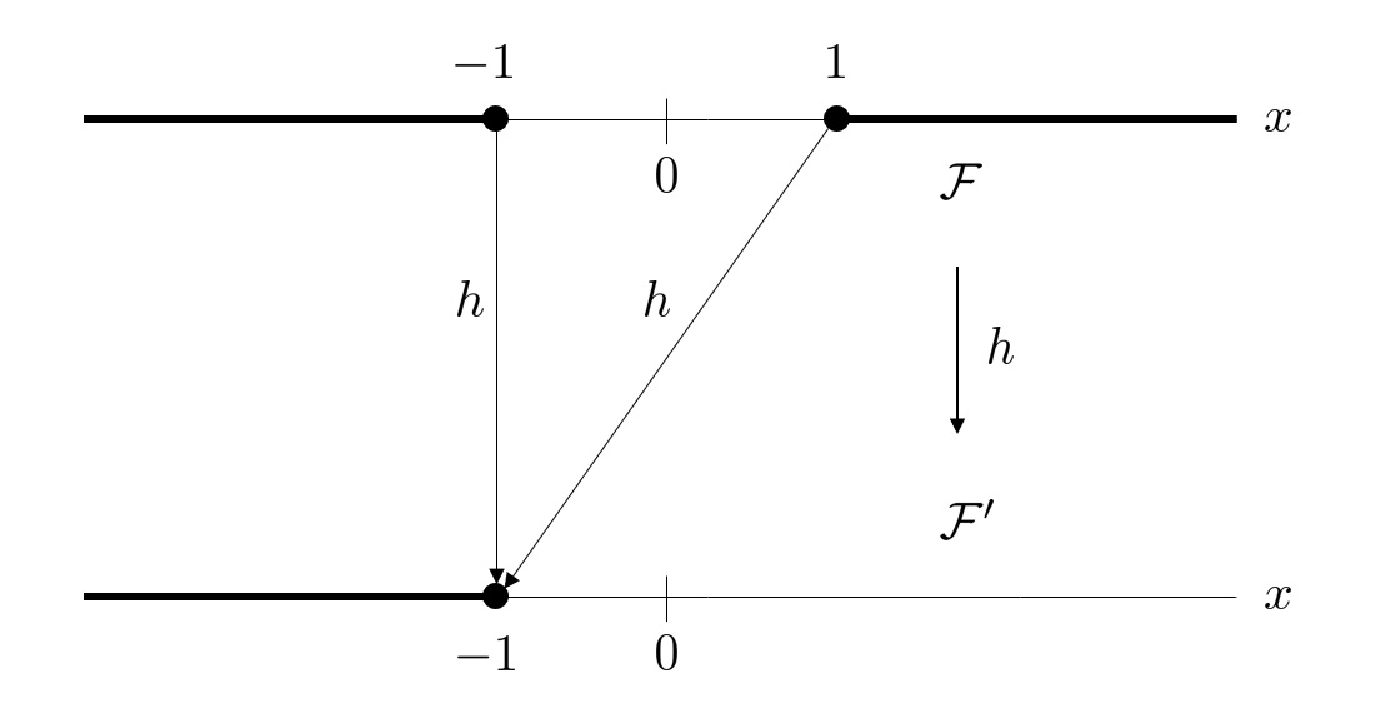}
\caption{\cref{ex:1}: Feasible sets, $\mathcal F$ and ${\mathcal F}'$, and mapping $h$.}\label{fig:0}
\end{figure}
\end{example}

\begin{example}\label{ex:2}
Let us consider the optimization problem
\begin{align}
&\min _{x_1 \in {\mathbb R},x_2 \in {\mathbb R}} \quad f(x_1 ,x_2 ) = x_1^2 + x_1 x_2^2\label{opt:1}\\
&{\rm{s}}{\rm{.t}}.\quad \Phi_1 (x_1 ,x_2 ) = 1-x_1 x_2^2 \leq 0, 1 - x_1  \le 0\nonumber
\end{align}
where ${\mathcal D} = {\mathbb R}^2$, and $x_1,x_2 \in {\mathbb R}$ are decision variables. For this problem, the constrained optimal solution is $x_1^* = x_2^* = 1$, the corresponding primal optimal value is $p^* = f(x^*) = 2$, and the corresponding feasible set is
\begin{align*}
{\mathcal F}: =& \{ (x_1 ,x_2 ) \in {\mathbb R}^2 :\Phi_1 (x) \le 0,\Phi_2 (x) \le 0\}\\
 =& \{ (x_1 ,x_2 ) \in {\mathbb R}^2 :x_1  \ge 1,x_1 x_2^2  \ge 1\},
\end{align*}
which is non-convex as shown in~\cref{fig:0b}.
\begin{figure}[h!]
\centering\includegraphics[width=9cm,height=7cm]{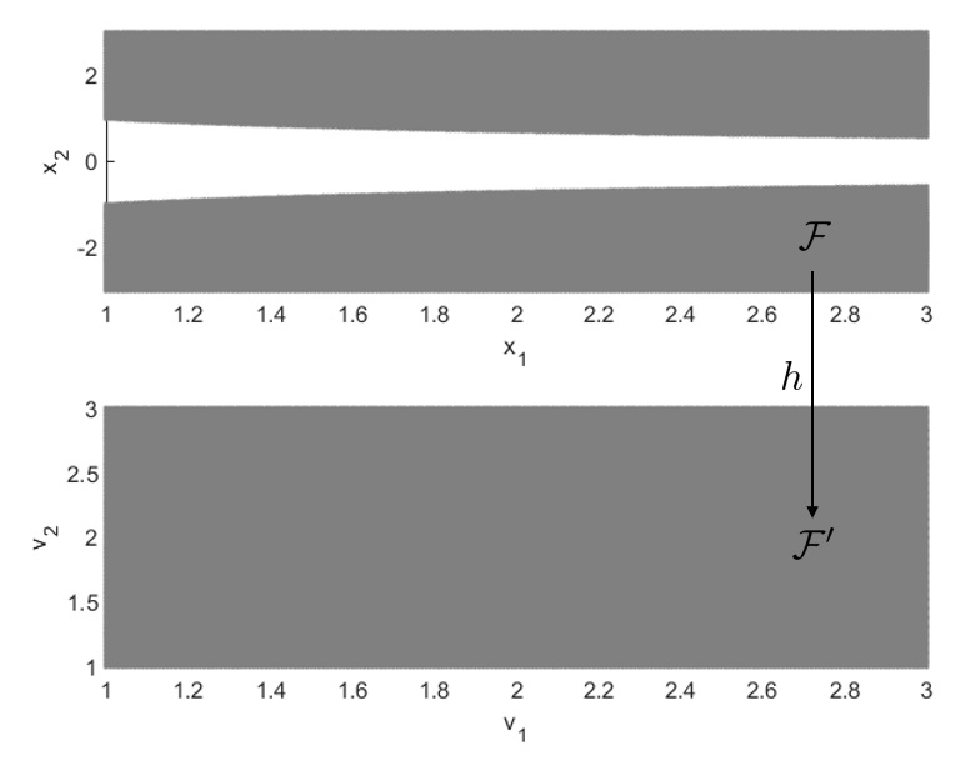}
\caption{\cref{ex:2}. Feasible sets, $\mathcal F$ and ${\mathcal F}'$, and mapping $h$.}\label{fig:0b}
\end{figure}

Now, let us consider the mapping
\begin{align*}
x = \left[ {\begin{array}{*{20}c}
   {x_1 }  \\
   {x_2 }  \\
\end{array}} \right] \mapsto h\left( {\left[ {\begin{array}{*{20}c}
   {x_1 }  \\
   {x_2 }  \\
\end{array}} \right]} \right) = \left[ {\begin{array}{*{20}c}
   {x_1 }  \\
   {x_1 x_2^2 }  \\
\end{array}} \right],
\end{align*}
and the corresponding change of variables
\begin{align*}
v = \left[ {\begin{array}{*{20}c}
   {v_1 }  \\
   {v_2 }  \\
\end{array}} \right] = h\left( {\left[ {\begin{array}{*{20}c}
   {x_1 }  \\
   {x_2 }  \\
\end{array}} \right]} \right) = \left[ {\begin{array}{*{20}c}
   {x_1 }  \\
   {x_1 x_2^2 }  \\
\end{array}} \right].
\end{align*}

Using this transformation, a convexification is given by
\begin{align}
&\min _{v_1  \in {\mathbb R},v_2  \in {\mathbb R}} \quad f'(v_1 ,v_2 ) = v_1^2 + v_2\label{eq:4}\\
&{\rm{s}}{\rm{.t}}.\quad \Phi_1' (v_1 ,v_2 ) = 1-v_2 \le 0,\quad \Phi_2' (v_1 ,v_2 ) = 1 - v_1  \le 0,\nonumber
\end{align}
where $h({\mathcal D}) = h({\mathbb R}^2 ) = {\mathbb R}^2$. The corresponding optimal solution is $v_1^*  = v_2^* = 1$, and the corresponding feasible set is
\begin{align*}
{\mathcal F}' =& \{ (v_1 ,v_2 ) \in {\mathbb R}^2 :\Phi _1 '(v) \le 0,\Phi _2 '(v) \le 0\}\\
=& \{ (v_1 ,v_2 ) \in {\mathbb R}^2 :v_1  \ge 1,v_2  \ge 1\},
\end{align*}
which is convex. Over $(v_1,v_2) \in {\mathcal F}'$, $v_1 \geq 1$ is invertible, and the inverse mapping, $h^{ - 1}: {\mathcal F}' \to {\mathcal F}$, is given by
\[
h^{ - 1}(v) = q(v) \left( {\left[ {\begin{array}{*{20}c}
   {v_1 }  \\
   {v_2 }  \\
\end{array}} \right]} \right) = \left[ {\begin{array}{*{20}c}
   {v_1 }  \\
   {v_2 /v_1 }  \\
\end{array}} \right] = \left[ {\begin{array}{*{20}c}
   {x_1 }  \\
   {x_2 }  \\
\end{array}} \right].
\]

Therefore, $h: {\mathcal F} \to {\mathcal F}'$ is a bijection (hence, a surjection),
and,~\eqref{eq:4} is a lossless convexification of~\eqref{opt:1} by~\cref{def:lossless-convexification}.
\end{example}

An implication of~\cref{def:lossless-convexification} is that solutions of~\cref{prob:4} have a surjective correspondence to solutions of~\cref{prob:1}. Therefore, even if~\cref{prob:1} is nonconvex, its solutions can be found from the convex~\cref{prob:4}.
Moreover, another property is that the existence of such a lossless convexification ensures strong duality of the original~\cref{prob:1} (with the Slater's condition). This result will be presented in the next subsection.

\subsection{Strong duality}
In this subsection, we investigate a relation between the lossless convexification and strong duality. Some preliminary definitions are first introduced below. Associated with~\cref{prob:1}, define the set
\begin{align*}
{\mathcal G}: =& \{ (\Phi _1 (x),\Phi _2 (x), \ldots ,\Phi _N (x),f(x))\\
& \in {\mathbb S}^n  \times {\mathbb S}^m  \times  \cdots  \times {\mathbb S}^m  \times {\mathbb R}:x \in {\mathcal D}\} ,
\end{align*}
which is called the graph of the constrained optimization problem in~\cref{prob:1}. The corresponding epigraph form~\cite{Boyd2004} is defined as
\begin{align}
{\mathcal A}: =& \{ (\bar U,t):\exists x \in {\mathcal D},\Phi_i(x) \preceq U_i,\forall i\in \{1,2,\ldots,N \}, f(x) \le t\}.\label{eq:23}
\end{align}
where $\bar U: = (U_1 ,U_2 , \ldots ,U_N )$. Note that $\mathcal A$ includes all the points in $\mathcal G$, as well as points that are `worse,' i.e., those with larger objective or inequality constraint function values~\cite{Boyd2004}. In other words, ${\mathcal A}$ can be also expressed as
\[
{\mathcal A} = {\mathcal G} + (\underbrace {{\mathbb S}_ + ^m  \times {\mathbb S}_ + ^m  \times  \cdots {\mathbb S}_ + ^m }_{N - {\rm{times}}} \times {\mathbb R}_ +  ),
\]
where ${\mathbb R}_ +$ is the set of nonnegative real numbers and `$+$' above is Minkowski sum. Similarly, for the convexified problem in~\cref{prob:4}, we define the graph and epigraph form as
\begin{align*}
{\mathcal G}': =& \{ (\Phi _1' (v),\Phi _2' (v), \ldots ,\Phi _N' (v),f'(v))\\
& \in {\mathbb S}^{m'}  \times {\mathbb S}^{m'}  \times  \cdots  \times {\mathbb S}^{m'}  \times {\mathbb R}:v \in h({\mathcal D})\},
\end{align*}
and
\begin{align*}
{\mathcal A}':=& \{ (\bar U,t): \exists v \in h({\cal D}),\\
&\Phi _i '(v) \preceq U_i ,\forall i \in \{ 1,2, \ldots ,N\} ,f'(v) \le t\}
\end{align*}
respectively. Note that since $\Phi _i', i\in \{1,2,\ldots, N \}$, and $f'$ are convex, so is ${\mathcal A}'$ as well.
Now, we are in position to present the main result.
\begin{theorem}[Strong duality]\label{thm:strong-duality2-matrix}
Suppose that~\cref{prob:4} is a lossless convexification of~\cref{prob:1}. If~\cref{prob:1} satisfies the Slater's condition, then strong duality holds for~\cref{prob:1}.
\end{theorem}
\begin{proof}
Suppose that~\cref{prob:4} is a lossless convexification of~\cref{prob:1}, and~\cref{prob:1} satisfies the Slater's condition.
By definition, the primal optimal value satisfies
\[
p^*  = \inf \{ t \in {\mathbb R}:(\underbrace {0,0, \ldots ,0}_{N - {\rm{times}}},t) \in {\mathcal A}\} ,
\]
where $0$ denotes a zero matrix with compatible dimensions.

Using the fact that $h: {\mathcal F} \to {\mathcal F}'$ is a surjection, we can prove the following claim.
\begin{claim}
We have
\begin{align*}
p^*  =& \inf \{ t \in {\mathbb R}:(\underbrace {0,0, \ldots ,0}_{N - {\rm{times}}},t) \in {\mathcal A}\} \\
 =& \inf \{ t \in {\mathbb R}:(\underbrace {0,0, \ldots ,0}_{N - {\rm{times}}},t) \in {\mathcal A}^\prime  \} ,
\end{align*}
where $0$ denotes a zero matrix with compatible dimensions.
\end{claim}
{\bf Proof of~Claim~1}: We will prove the following identity: $\{ t \in {\mathbb R}:(0,t) \in {\mathcal A}\}  = \{ t\in {\mathbb R}:(0,t) \in {\mathcal A}'\}$.
Suppose that $(0,t) \in {\mathcal A}$, i.e., $\Phi_i (x) \preceq 0 ,f(x) \le t$ holds for some $x \in {\mathcal D}$. Then, $\Phi_i '(h(x)) = \Phi _i '(v) \preceq 0,f'(h(x))= f'(v) \le t$ also holds. This implies that $(0,t) \in {\mathcal A}'$ holds. On the other hand, suppose $\Phi _i' (v) \preceq 0 ,f'(v) \le t$ holds for some $v \in h({\mathcal D})$. Then, since $h: {\mathcal F} \to {\mathcal F}'$ is a surjection, there exists a mapping, $q:{\mathcal F}'\to {\mathcal F}$, such that $h(q(v))=v$, and hence, it follows that
\[
\Phi _i '(v) = \Phi _i '((h \circ q )(v)) = (\Phi _i ' \circ h)(q (v)) = \Phi _i (x) \preceq 0,
\]
and
\[
f'(v) = f'((h \circ q )(v)) = (f' \circ h)(q(v)) = f(x) \le t.
\]

Therefore, $\{ t\in {\mathbb R}:(0,t) \in {\mathcal A}\}  = \{ t\in {\mathbb R}:(0,t) \in {\mathcal A}'\}$, and hence $p^*  = \inf \{ t\in {\mathbb R} :(0,t) \in {\mathcal A}\}  = \inf \{ t\in {\mathbb R} :(0,t) \in {\mathcal A}'\}$. This completes the proof. $\blacksquare$

Using similar arguments, we can also prove the following key result.
\begin{claim}\label{claim:1}
Let us divide ${\mathcal A}$ into the two sets
\begin{align*}
{\mathcal A}_1 :=\{ (\bar U,t) \in {\mathcal A}:U_i \preceq 0, \forall i\in \{1,2,\ldots,N \}\}
\end{align*}
and
\begin{align*}
{\mathcal A}_2:=\{ (\bar U,t) \in {\mathcal A}:U_i \npreceq 0 \,\, \text{for some}\,\, i\in \{1,2,\ldots,N \}\},
\end{align*}
so that ${\mathcal A}_1 \cup {\mathcal A}_2 = {\mathcal A}$.
Moreover, divide ${\mathcal A}'$ into the two sets
\begin{align*}
{\mathcal A}_1' :=\{ (\bar U,t) \in {\mathcal A}':U_i \preceq 0,\forall i\in \{1,2,\ldots,N \}\}
\end{align*}
and
\begin{align*}
{\mathcal A}_2' :=\{ (\bar U,t) \in {\mathcal A}':U_i \npreceq 0,\,\, \text{for some} \,\, i\in \{1,2,\ldots,N \}\}
\end{align*}
so that ${\mathcal A}_1' \cup {\mathcal A}_2' = {\mathcal A}'$.
Then, the following statements hold true:
\begin{enumerate}
\item ${\mathcal A}_1 = {\mathcal A}_1'$,

\item ${\mathcal A}_2 \subseteq {\mathcal A}_2'$,

\item ${\mathcal A} \subseteq {\mathcal A}'$.
\end{enumerate}
\end{claim}
{\bf Proof of Claim~2}: We first prove the identity ${\mathcal A}_1 = {\mathcal A}_1'$. This part is similar to the proof of~{\bf Claim~1}.
Suppose that $U_i \preceq 0, \forall i\in \{1,2,\ldots,N \}$ holds, and $(\bar U,t) \in {\mathcal A}$, i.e., $\Phi_i (x) \preceq U_i ,f(x) \le t$ holds for some $x \in {\mathcal D}$. Then, $\Phi_i '(h(x)) = \Phi _i '(v) \preceq U_i ,f'(h(x))= f'(v) \le t$ also holds. This implies that $(\bar U,t) \in {\mathcal A}'$ holds. On the other hand, suppose $\Phi _i' (v) \preceq U_i \preceq 0 ,f'(v) \le t$ holds for some $v \in h({\mathcal D})$. Then, since $h: {\mathcal F} \to {\mathcal F}'$ is a surjection, there exists a mapping, $q:{\mathcal F}'\to {\mathcal F}$, such that $h(q(v))=v$, and hence, it follows that
\[
\Phi _i '(v) = \Phi _i '((h \circ q )(v)) = (\Phi _i ' \circ h)(q (v)) = \Phi _i (x) \preceq U_i ,
\]
and
\[
f'(v) = f'((h \circ q )(v)) = (f' \circ h)(q(v)) = f(x) \le t.
\]
Therefore, ${\mathcal A}_1 = {\mathcal A}_1'$ holds.

On the other hand, if $U_i \npreceq 0$ for some $i\in \{1,2,\ldots,N \}$, then ${\mathcal A}_2 \subseteq {\mathcal A}_2'$ is satisfied.
In particular, suppose that $\Phi_i (x)\preceq U_i ,f(x) \le t$ holds for some $x \in {\mathcal D}$. Then, $\Phi_i'(h(x)) = \Phi _i '(v) \preceq U_i ,f'(h(x))= f'(v) \le t$ also holds. Therefore, $(\bar U,t) \in {\mathcal A}'$, which implies ${\mathcal A}_2 \subseteq {\mathcal A}_2'$. The reverse does not hold in general. To see it, suppose $\Phi _i' (v) \preceq U_i ,f'(v) \le t$ holds for some $v \in h({\mathcal D})$. Then, since there is no guarantee that $h:{\mathcal F} \to {\mathcal F}'$ is a surjection when $U_i \npreceq 0$ for some $i\in \{1,2,\ldots,N \}$, we cannot ensure that there exists $x$ such that $\Phi _i (x) \preceq U_i ,f(x) \le t$. The statement~3) is implied by the statement~1) and statement~2). This completes the proof. $\blacksquare$

Next, let us return to our main concern. To prove~\cref{thm:strong-duality2-matrix}, we define the set
\[
{\mathcal B}: = \{ (0,s):s < p^* \}.
\]
Since ${\mathcal A}'$ and $\mathcal B$ are convex, there exists a separating hyperplane that separates $\mathcal A'$ and $\mathcal B$ by the separating hyperplane theorem in~\cite[sec.~2.5]{Boyd2004}, as shown in~\cref{fig:2}.
\begin{figure}[h!]
\centering\includegraphics[width=9cm,height=6cm]{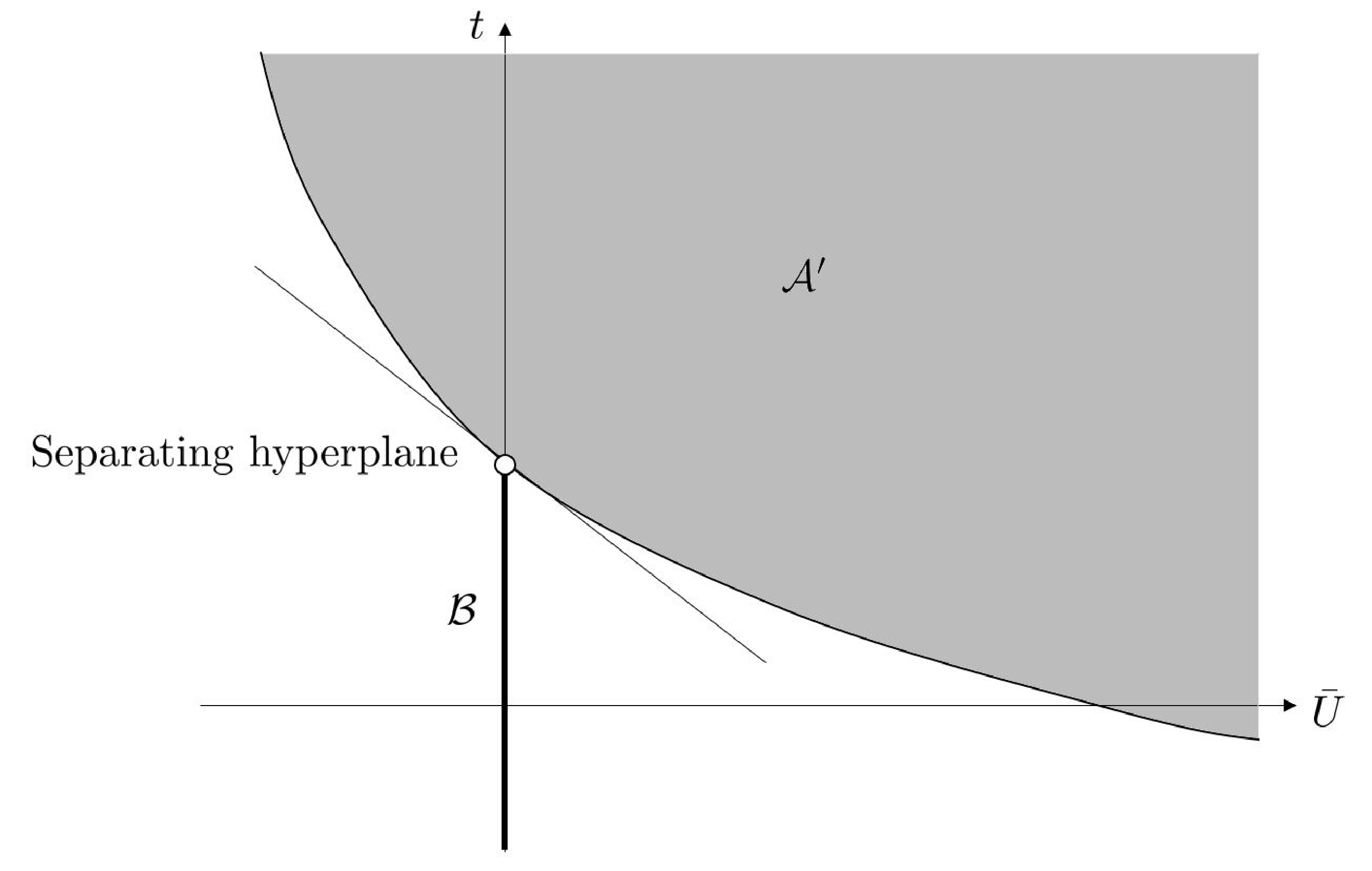}
\caption{Two convex sets, ${\mathcal A}'$ (shaded region) and $\mathcal B$, and a separating hyperplane. }\label{fig:2}
\end{figure}

In particular, the existence of a separating hyperplane implies that there exists $(\bar \Lambda,\mu)\neq 0$ and $\alpha \in {\mathbb R}$ such that
\begin{align}
(\bar U,t) \in {\mathcal A}' \Rightarrow \sum\limits_{i = 1}^N {Tr(\Lambda _i U_i )}  + \mu t \ge \alpha \label{eq:26}
\end{align}
and
\[
(\bar U,t) \in {\mathcal B} \Rightarrow \sum\limits_{i = 1}^N {Tr(\Lambda _i U_i )}  + \mu t \le \alpha.
\]

Using the inclusion ${\mathcal A} \subseteq {\mathcal A}'$ in~{\bf Claim~2}, it follows from~\eqref{eq:26} that
\begin{align}
(\bar U,t) \in {\mathcal A} \Rightarrow \sum\limits_{i = 1}^N {Tr(\Lambda _i U_i )}  + \mu t \ge \alpha
\label{eq:2b}
\end{align}
and
\begin{align}
(\bar U,t) \in {\mathcal B} \Rightarrow \sum\limits_{i = 1}^N {Tr(\Lambda _i U_i )}  + \mu t \le \alpha \label{eq:3b}
\end{align}

From~\eqref{eq:2b}, we conclude that $\Lambda_i \succeq 0,i\in \{1,2,\ldots,N \}$ and $\mu\geq 0$. Otherwise $\sum\limits_{i = 1}^N {Tr(\Lambda _i U_i )}  + \mu t$ is unbounded below over ${\mathcal A}$, contradicting~\eqref{eq:2b}.
The condition~\eqref{eq:3b} simply means that $\mu t \leq \alpha$ (since $\bar U =0$) for all $t < p^*$, and hence, $\mu p^* \leq \alpha$. Together with~\eqref{eq:2b}, we conclude that for any $x \in \mathcal D$
\begin{align}
\sum\limits_{i = 1}^N {Tr(\Lambda _i \Phi _i (x))}  + \mu f(x) \ge \alpha  \ge \mu p^* \label{eq:4b}
\end{align}
Assume that $\mu >0$. In that case we can divide~\eqref{eq:4b} by $\mu$ to obtain
\[
L\left( {x,\frac{{\bar \Lambda }}{\mu }} \right) = \sum\limits_{i = 1}^N {Tr\left( {\Phi _i (x)\frac{{\Lambda _i }}{\mu }} \right)}  + f(x) \ge p^*
\]
for all $x \in \mathcal D$, from which it follows, by minimizing over $x$, that $g\left( {\frac{\bar \Lambda }{{\mu }}} \right) \ge p^*$. By weak duality we have $g\left( {\frac{\bar \Lambda }{{\mu}}} \right) \le p^*$, so in fact $g\left( {\frac{\bar \Lambda }{{\mu}}} \right) = p^*$. This shows that strong duality holds, and that the dual optimum is attained, at least in the case when $\mu >0$.

Now consider the case $\mu =0$. From~\eqref{eq:2b}, we conclude that for all $x\in \mathcal D$,
\[
\sum\limits_{i = 1}^N {Tr(\Lambda _i \Phi _i (x))}  \ge 0
\]
Assume that $\tilde x$ is the point that satisfies the Slater condition, i.e., $\Phi_i(\tilde x) \prec 0,i\in \{1,2,\ldots,N \}$. Then, we have $\sum\limits_{i = 1}^N {Tr(\Lambda _i \Phi _i (\tilde x))}  \ge 0$. Since $\Phi_i(\tilde x) <0,i\in \{1,2,\ldots,N \}$ and $\Lambda \succeq 0$, we conclude that $\Lambda_i = 0,i\in \{1,2,\ldots,N \}$. Therefore, $\Lambda_i = 0,i\in \{1,2,\ldots,N \}$ and $\mu = 0$, which contradicts $(\bar \Lambda,\mu)\neq 0$. Intuition of the overall ideas are illustrated in~\cref{fig:1}. ${\mathcal A}'$ is convex, while $\mathcal A$ may not be in general. $\mathcal A$ and ${\mathcal A}'$ are identical over $U_i \preceq 0,i\in \{1,2,\ldots,N \}$, while ${\mathcal A}'$ includes $\mathcal A$ when $U_i \npreceq 0,i\in \{1,2,\ldots,N \}$. Overall, the boundary of ${\mathcal A}'$ is a convex envelope of $\mathcal A$, and both ${\mathcal A}'$ and ${\mathcal A}$ share the same separating hyperplane.
\begin{figure}[h!]
\centering\includegraphics[width=9cm,height=6cm]{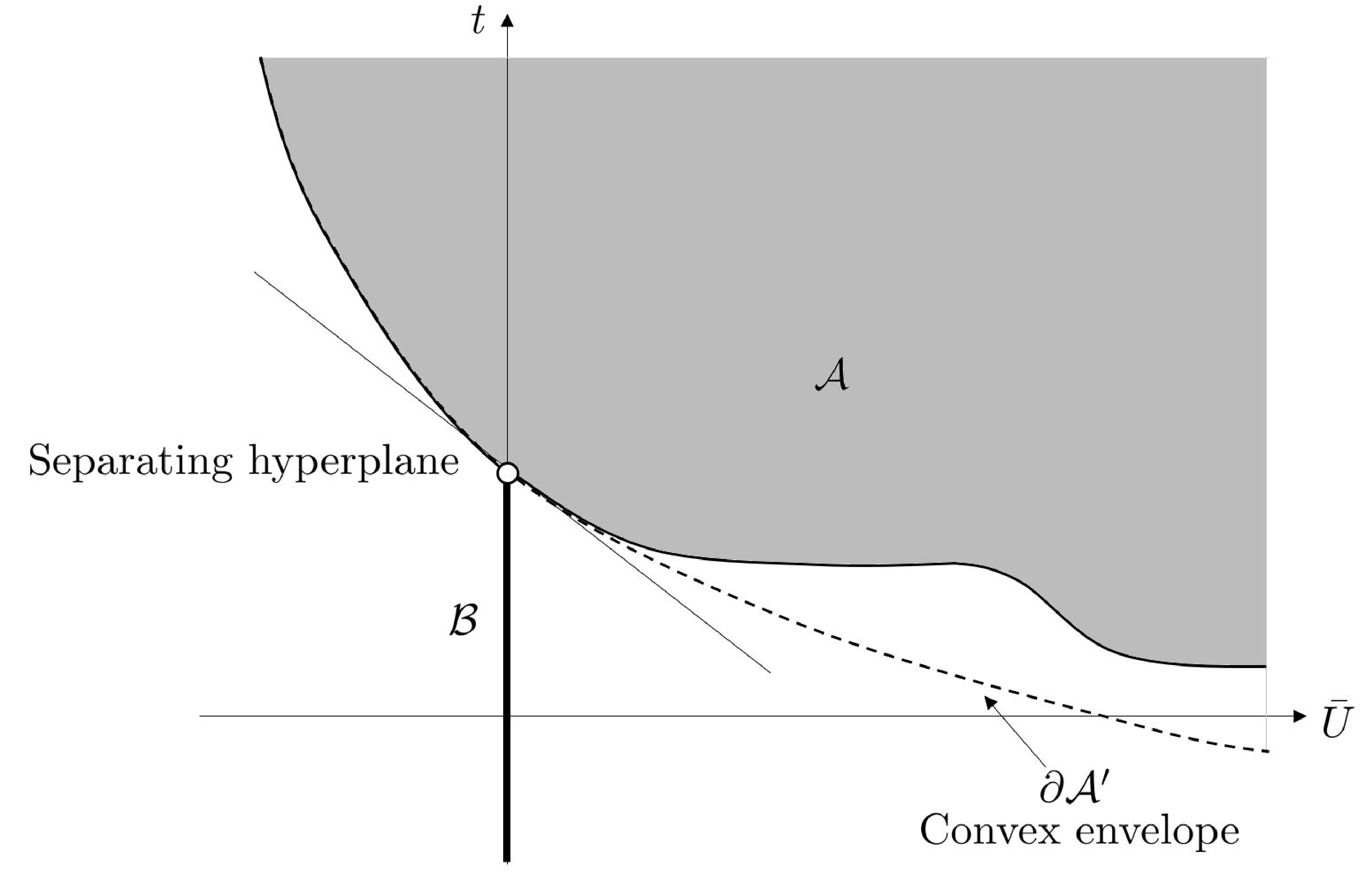}
\caption{Sets $\mathcal A$ (shaded region), $\mathcal B$, the separating hyperplane, and the boundary of ${\mathcal A}'$ (dashed line), denoted by  $\partial {\mathcal A}'$}\label{fig:1}
\end{figure}

\end{proof}

In the sequel, several examples are given to illustrate the ideas of~\cref{thm:strong-duality2-matrix}.
\begin{example}\label{ex:3}
Let us consider~\cref{ex:1} again. The optimization in~\eqref{eq:3} satisfies the Slater's condition because with $x= 2$, we have $\Phi _1 (x) =  -3 < 0$. Moreover,~\eqref{eq:3} a lossless convexification in~\eqref{eq:5}. Therefore, strong duality holds by~\cref{thm:strong-duality2-matrix}. To prove it directly, consider the corresponding Lagrangian function
\[
L(x,\lambda) = x^2  + \lambda ( - x^2  + 1),
\]
where $\lambda \geq 0$ is the Lagrangian multiplier, and the dual function
\begin{align*}
g(\lambda ) =& \inf _{x \in {\mathbb R}} L(x,\lambda )\\
 =& \inf _{x \in {\mathbb R}} \{ x^2  + \lambda ( - x^2  + 1)\}.
\end{align*}

To obtain a more explicit form of the dual function, one can observe that this dual function is bounded below in $x$ only when $\lambda \leq 1$. For $\lambda > 1$, we have $g(\lambda )$ unbounded below in $x$.
Therefore, we can conclude that the dual function is
\[
g(\lambda ) = \left\{ {\begin{array}{*{20}c}
   { - \infty ,\quad {\rm{if}}\,\,\lambda  > 1}  \\
   {\lambda ,\quad {\rm{if}}\,\,0 \le \lambda  \le 1}  \\
\end{array}} \right.
\]

Then, the corresponding dual optimal value is $d^* = \sup _{\lambda  \ge 0} g(\lambda ) = 1=p^*$. Therefore, strong duality holds.
\end{example}

\begin{example}
Let us consider~\cref{ex:2} again. The optimization~\eqref{opt:1} satisfies the Slater's condition because with $x_1 = 2$ and $x_2 = 1$, we have $\Phi_1 (x_1 ,x_2 ) = -3 < 0, \Phi_1 (x_1 ,x_2 ) = -1 < 0$. Moreover,~\eqref{opt:1} admits the lossless convexification~\eqref{eq:4}. Therefore, by~\cref{thm:strong-duality2-matrix}, strong duality holds for~\eqref{opt:1}. The corresponding Lagrangian function is
\[
L(x,\bar \Lambda ) = x_1^2  + x_1 x_2^2  + \lambda _1 (1 - x_1 x_2^2 ) + \lambda _2 (1 - x_1 ),
\]
where $\bar \Lambda =  (\lambda _1, \lambda _2)$ is the Lagrangian multipliers or the dual variables. Let us manually check if strong duality really holds. The dual function is given by
\begin{align*}
g(\bar \Lambda ) =& \inf_{x_1\in {\mathbb R} ,x_2  \in {\mathbb R}} L(x,\bar \Lambda )\\
=& \inf _{x_1\in {\mathbb R} ,x_2 \in {\mathbb R}} \{ x_1^2  + x_1 x_2^2  + \lambda _1 (1 - x_1 x_2^2 ) + \lambda _2 (1 - x_1 )\}\\
=& \inf _{x_1 \in {\mathbb R},x_2 \in {\mathbb R}} \{ x_1^2  - \lambda _2 x_1  + (1 - \lambda _1 )x_1 x_2^2  + \lambda _1  + \lambda _2 \}.
\end{align*}
To obtain an explicit form of the dual function, we first find extrema of $L(x,\bar \Lambda )$ by checking the following first-order optimality conditions:
\begin{align}
&\frac{d}{{dx_1 }}\{ x_1^2  - \lambda _2 x_1  + (1 - \lambda _1 )x_1 x_2  + \lambda _1  + \lambda _2 \}\nonumber\\
=& 2x_1  - \lambda _2  + (1 - \lambda _1 )x_2^2\nonumber \\
=& 0\label{eq:1}
\end{align}
and
\begin{align}
&\frac{d}{{dx_2 }}\{ x_1^2  - \lambda _2 x_1  + (1 - \lambda _1 )x_1 x_2^2  + \lambda _1  + \lambda _2 \}\nonumber\\
=& 2(1 - \lambda _1 )x_1 x_2  = 0.\label{eq:2}
\end{align}

To satisfy~\eqref{eq:2}, in the case $1 - \lambda _1  \ne 0$, we have $x_1 = 0$ or $x_2 = 0$.
If $x_1=0$, then $x_2  = \sqrt {\frac{{\lambda _2 }}{{1 - \lambda _1 }}}$ provided that $1 > \lambda _1$. Therefore, the dual function satisfies
\begin{align*}
\mathop {\sup }\limits_{\lambda _1 ,\lambda _2  \ge 0} g(\bar \Lambda ) =& \mathop {\sup }\limits_{\lambda _1 ,\lambda _2  \ge 0} \left\{ {x_1^2  - \lambda _2 x_1  + (1 - \lambda _1 )x_1 x_2^2  + \lambda _1  + \lambda _2 } \right\}\\
 =& \mathop {\sup }\limits_{\lambda _1 ,\lambda _2  \ge 0} \left\{ {\lambda _1  + \lambda _2 :1 > \lambda _1 } \right\}\\
 =& \infty.
\end{align*}
Therefore, this case is discarded.
If $x_2=0$, then $x_1 = \lambda_2/2$, and the dual problem is
\begin{align*}
\mathop {\sup }\limits_{\lambda _1 ,\lambda _2  \ge 0} g(\bar \Lambda ) =& \mathop {\sup }\limits_{\lambda _1 ,\lambda _2  \ge 0} \left\{ {x_1^2  - \lambda _2 x_1  + (1 - \lambda _1 )x_1 x_2^2  + \lambda _1  + \lambda _2 } \right\}\\
=& \mathop {\sup }\limits_{\lambda _1 ,\lambda _2  \ge 0} \left\{ { - \frac{{\lambda _2^2 }}{4} + \lambda _1  + \lambda _2 :1 - \lambda _1  \ne 0} \right\}\\
=& \infty.
\end{align*}
Therefore, this case is also discarded, and $1 - \lambda _1  = 0$ should be hold. In this case, we have
\[
x_1  = \frac{{\lambda _2 }}{2},\quad x_2  \in {\mathbb R},
\]
and the corresponding dual problem is
\begin{align*}
&\sup _{\lambda _1 ,\lambda _2  \ge 0} g(\bar \Lambda )\\
 =& \sup _{\lambda _1 ,\lambda _2  \ge 0} \left\{ { - \frac{{\lambda _2^2 }}{4} + \lambda _1  + \lambda _2 :\,1 - \lambda _1  = 0} \right\}\\
  =& 2,
\end{align*}
whose optimal dual value is $d^* = 2$. Therefore, $d^* = 2 = p^*$, and strong duality holds.
\end{example}

In this section, the notion of the lossless convexification has been introduced, and a relation between the lossless convexification and strong duality has been established. In the next section, we will present a formulation of the state-feedback problem, and prove the corresponding strong duality using the results in this paper.
\begin{remark}
We note that simple examples are presented until now as a form of sanity check, yet the results have broader applicability to general optimization problems. The subsequent section delves into some applications in the field of control engineering.
\end{remark}
\begin{remark}
Many control design challenges, including the linear quadratic regulator (LQR), linear quadratic Gaussian (LQG), stabilization problem, and the $H_\infty$ control design problem, can be expressed as constrained optimization tasks aiming to minimize nonconvex objective functions. These problems often benefit from the application of results that demonstrate strong duality in the corresponding optimization tasks. Specifically, these optimization problems satisfy the Karush-Kuhn-Tucker (KKT) conditions, which are both necessary and sufficient when the principles of strong duality are applied. This understanding opens up new algorithmic avenues for solving the original problems, such as the implementation of primal-dual algorithms, leading to effective solutions. Recent studies, for example, \cite{lee2019primal, farjadnasab2022model, clarke2022low, li2022model, esmzad2023maximum}, have developed primal-dual algorithms for optimal control problems, with the concept of strong duality playing a pivotal role. The results presented in this paper are crucial for the theoretical analysis of these methods and offer a general and straightforward framework for establishing strong duality in a variety of optimization problems found in engineering.

The application of strong duality not only enriches the theoretical understanding by confirming the solvability of the KKT conditions but also opens new theoretical research pathways. For instance, \cite{lee2019primal} reveals that the solution to the KKT conditions corresponds to the solution of the Riccati equations, thereby establishing an intriguing theoretical connection between the KKT solution and the Riccati equations through the concept of strong duality.
\end{remark}

\begin{remark}
As previously mentioned, the research presented in~\cite{sun2021learning,mohammadi2019global} establishes a critical link between the concept of lossless convexification and the gradient dominance condition in gradient-based policy search algorithms. These studies demonstrate that in policy optimization problems within control and reinforcement learning, the presence of lossless convexification ensures the gradient dominance property of the underlying objective function, despite its inherent non-convexity.
It is anticipated that the proposed strong duality property shares some connections with the properties explored in~\cite{sun2021learning,mohammadi2019global}. Understanding these connections could provide additional insights, enhancing our comprehension and potentially aiding in the development of new methodologies for solving various optimization problems.
\end{remark}

\section{Example: state-feedback stabilization}
In this section, we present a control problem where the proposed results can be applied.

\subsection{Continuous-time case}
Let us consider the continuous-time linear time-invariant (LTI) system
\begin{align}
\dot x(t) = Ax(t) + Bu(t),\quad x(0) \in {\mathbb R}^n,\label{eq:linear-system}
\end{align}
where $A \in {\mathbb R}^{n\times n}$, $B \in {\mathbb R}^{n\times m}$, $t \geq 0$ is the time, $x(t) \in {\mathbb R}^n$ is the state
vector, and $u(t) \in {\mathbb R}^m$ is the input vector. One of the most fundamental problems for this LTI system is the state-feedback stabilization problem, which is designing a state-feedback control input, $u(t) = F x(t)$, where $F \in {\mathbb R}^{m\times n}$ is called the state-feedback gain matrix, such that the closed-loop system, $\dot x(t) = Ax(t) + Bu(t) = (A + BF)x(t)$, is asymptotically stable~\cite{khalil2002nonlinear}. Some related concepts in the standard linear system theory~\cite{chen1995linear} are briefly reviewed first.
\begin{definition}[{\cite[page~55]{khalil2015nonlinear}}]
A matrix $A\in {\mathbb R}^{n\times n}$ is said to be Hurwitz when all the eigenvalues of $A$ have strictly negative real parts.
\end{definition}
\begin{lemma}[{\cite[page~55]{khalil2015nonlinear}}]\label{thm:stability}
The system $\dot x(t) = Ax(t)$ is asymptotically stable if and only if $A$ is Hurwitz.
\end{lemma}

The Laypunov method provides a way to check the stability of $\dot x(t) = Ax(t)$ and the Hurwitz property of $A$ without actually calculating the eigenvalues.
\begin{lemma}[{\cite[Thm.~5.5]{chen1995linear}}]\label{thm:Lyapunov-theorem}
A matrix $A \in {\mathbb R}^{n\times n}$ is Hurwitz if and only if there exists a symmetric matrix $P\in {\mathbb S}^n$ such that
\[
P \succ 0,\quad A^T P + PA \prec 0,
\]
where $P \in {\mathbb S}^n$ is called the Lyapunov matrix.
\end{lemma}

From~\cref{thm:stability}, the state-feedback problem can be formally written as follows.
\begin{problem}[State-feedback stabilization problem]\label{problem:stabilization}
Find a state-feedback gain $F \in {\mathbb R}^{m\times n}$ such that $A+BF$ is Hurwitz.
\end{problem}

Using~\cref{thm:Lyapunov-theorem}, the state-feedback problem can be formulated as the feasibility problem with matrix inequalities (Lyapunov inequalities)
\[
P \succeq \varepsilon I,\quad (A + BF)^T P + P(A + BF)   \preceq -\varepsilon I,
\]
where $\varepsilon>0$ is some fixed sufficiently small number, and the second inequality is non-convex (bilinear in the decision variables $F$ and $P$), i.e., the feasible set of the second inequality may be non-convex in general. The above results are well known in the control literature, and more comprehensive introductions of LMI-based control design processes can be found in~\cite{Boyd1994,el2000advances}. Here, the bilinear inequality is a special case of non-convexity inequalities, and implies that it is linear if the other variable is fixed, and vice versa. From standard results of linear system theory~\cite{chen1995linear}, it can be proved that the condition is equivalent to
\begin{align}
P \succeq \varepsilon I,\quad (A + BF) P + P(A + BF)^T   \preceq - \varepsilon I.\label{eq:7}
\end{align}

This is because the eigenvalues of $A + BF$ are identical to the eigenvalues of $(A + BF)^T$ by duality, and it implies that $(A + BF)^T$ is Hurwitz and by the Lyapunov theorem in~\cref{thm:Lyapunov-theorem}, $(A + BF)^T$ also admits a Lyapunov matrix $P\in {\mathbb S}^n$ such that~\eqref{eq:7} holds. The corresponding feasibility problem in~\cref{prob:1} can be converted to the equivalent optimization
\begin{align}
&\min_{P \in {\mathbb S}^n ,F \in {\mathbb R}^{m \times n} } \,\,\quad 0\label{eq:10}\\
&{\rm{s}}.{\rm{t}}.\quad \Phi _1 (x) = (A + BF)P + P(A + BF)^T  + \varepsilon I \preceq 0\nonumber\\
&\Phi _2 (x) = \varepsilon I - P \preceq 0,\nonumber
\end{align}
where
\begin{align*}
x = \left[ {\begin{array}{*{20}c}
   P  \\
   F  \\
\end{array}} \right].
\end{align*}
Note that the feasibility problem in~\eqref{eq:7} is equivalent to the optimization in~\eqref{eq:10} in the sense that their solutions are identical. We consider the optimization form in~\eqref{eq:10} to fit the problem into the optimization form in~\cref{prob:1}, and  this it not more than formality. Moreover, note that the problem in~\eqref{eq:10} only includes the inequality constraints, and the problem do not have the objective function, that is, it is a feasibility problem, which can be seen as a special case of the general optimization problems. Therefore, in order to convert the feasibility problem into the optimization problem format, the null objective $0$ has been used.
Next, consider the mapping
\begin{align*}
x = \left[ {\begin{array}{*{20}c}
   P  \\
   F  \\
\end{array}} \right] \mapsto h\left( {\left[ {\begin{array}{*{20}c}
   P  \\
   F  \\
\end{array}} \right]} \right) = \left[ {\begin{array}{*{20}c}
   P  \\
   {FP}  \\
\end{array}} \right],
\end{align*}
and the corresponding change of variables
\begin{align*}
v = \left[ {\begin{array}{*{20}c}
   P  \\
   M  \\
\end{array}} \right] = h\left( {\left[ {\begin{array}{*{20}c}
   P  \\
   F  \\
\end{array}} \right]} \right) = \left[ {\begin{array}{*{20}c}
   P  \\
   {FP}  \\
\end{array}} \right].
\end{align*}

Using this transformation, a convexification of~\eqref{eq:10} is
\begin{align*}
&\min_{P \in {\mathbb S}^n ,M \in {\mathbb R}^{m \times n} } \,\,\quad 0\\
&{\rm s.t.}\quad \Omega _1' (v) = (AP + BM) + (AP + BM)^T  + \varepsilon I \preceq 0,\\
&\Omega _2' (v) = \varepsilon I - P \preceq 0.
\end{align*}

We note that the above conversion from non-convex to convex problems via change of variables is very well-known and popular technique in the control literature, e.g.,~\cite{geromel2007h,kim2005new,ghaffari2022robust,de2022new}.
Over the feasible sets, ${\mathcal F}$ and ${\mathcal F}'$, $P$ is nonsingular. Therefore, $h: {\mathcal F} \to {\mathcal F}'$ is a bijection (and hence, a surjection), and the inverse mapping is given by
\[
h^{ - 1} \left( {\left[ {\begin{array}{*{20}c}
   P  \\
   M  \\
\end{array}} \right]} \right) = \left( {\left[ {\begin{array}{*{20}c}
   P  \\
   {MP^{ - 1} }  \\
\end{array}} \right]} \right),
\]
where $\left[ {\begin{array}{*{20}c}
   P  \\
   M  \\
\end{array}} \right] \in {\mathcal F}'$.
Therefore,~\eqref{eq:10} is initially nonconvex, but can be convexified using some transformations and manipulations. Lastly, we can prove that the original problem in~\eqref{eq:10} satisfies the Slater's condition under a mild assumption, and hence, it satisfies strong duality by~\cref{thm:strong-duality2-matrix}.
\begin{claim}\label{claim:4}
Suppose that $(A,B)$ is stabilizable. Then, there exists a sufficiently small $\varepsilon>0$, $P \in {\mathbb S}^n$ and $F\in {\mathbb R}^{m\times n}$ such that
\begin{align}
(A + BF)P + P(A + BF)^T  + \varepsilon I \preceq 0,\quad \varepsilon I - P \preceq 0.\label{eq:6}
\end{align}
\end{claim}
\begin{proof}
It is clear from Lyapunov theory in~\cref{thm:Lyapunov-theorem} that if $(A,B)$ is stabilizable, then there exists $P\succ 0$ and $F\in {\mathbb R}^{m\times n}$ such that $(A + BF)P + P(A + BF)^T \prec 0$. Then, there always exists a sufficiently small $\varepsilon>0$ such that~\eqref{eq:6} is satisfied. Therefore, the proof is completed.
\end{proof}
We note that the above claim is somewhat straightforward, yet it has not been formally reported in the literature. Therefore, it is still worth introducing formally for the convenience of presentation.

By~{\bf Claim~3}, the optimization~\eqref{eq:10} admits a strictly feasible solution.
This is because for the feasible $\varepsilon >0$ that satisfies~\eqref{eq:6}, we can divide it by $2$, and then with this $\varepsilon$, the inequalities in~\eqref{eq:10} are satisfied with strict inequalities.
In conclusion, the original problem~\eqref{eq:10} satisfies strong duality by~\cref{thm:strong-duality2-matrix}.

Now, let us consider the output vector
\[
y(t) =Cx(t) \in {\mathbb R}^p,
\]
where $C\in {\mathbb R}^{p\times n}$ is the output matrix. The static output-feedback control problem is designing a static output-feedback control input, $u(t) = F y(t)$, where $F \in {\mathbb R}^{m\times p}$ is called the static output-feedback gain matrix, such that the closed-loop system, $\dot x(t) = Ax(t) + Bu(t) = (A + BFC)x(t)$, is asymptotically stable~\cite{khalil2002nonlinear}. This problem is formally stated in the following.
\begin{problem}[Static output-feedback stabilization]\label{problem:static-output}
Find a feedback gain $F \in {\mathbb R}^{m\times p}$ such that $A+BFC$ is Hurwitz.
\end{problem}
The static output-feedback problem is much more challenging than the state-feedback problem, and is known to be non-convex and NP-hard~\cite{fu1997computational,fu2004pole,blondel1997np}. Using Lyapunov theory in~\cref{thm:Lyapunov-theorem} again, it can be rewritten by the non-convex optimization
\begin{align}
&\min_{P \in {\mathbb S}^n ,F \in {\mathbb R}^{m \times p} } \,\,\quad 0\label{prob:3}\\
&{\rm{s}}.{\rm{t}}.\quad \Phi _1 (x) = (A + BFC)P + P(A + BFC)^T  + \varepsilon I \preceq 0,\nonumber\\
&\Phi _2 (x) = \varepsilon I - P \preceq 0\nonumber,
\end{align}
where
\begin{align*}
x = \left[ {\begin{array}{*{20}c}
   P  \\
   F  \\
\end{array}} \right].
\end{align*}
Due to the matrix $C$, it is hard to find a lossless convexification in general.
This is because, the matrix $C$ is not invertible, and it separates matrices $F$ and $P$.
Therefore, it is hard to guarantee strong duality for~\eqref{prob:3}.

\subsection{Discrete-time case}
Let us consider the discrete-time linear time-invariant (LTI) system
\[
x(k + 1) = Ax(k) + Bu(k),\quad x(0) \in {\mathbb R}^n,
\]
where $A \in {\mathbb R}^{n\times n}$, $B \in {\mathbb R}^{n\times m}$, the integer $k \geq 0$ is the time, $x(k) \in {\mathbb R}^n$ is the state vector, and $u(k) \in {\mathbb R}^m$ is the input vector. Similar to the continuous-time case, we are interested in the state-feedback stabilization problem: design a state-feedback control input, $u(k) = F x(k)$, where $F \in {\mathbb R}^{m\times n}$ is called the state-feedback gain matrix, such that the closed-loop system, $x(k+1) = Ax(k) + Bu(k) = (A + BF)x(k)$, is asymptotically stable~\cite{khalil2002nonlinear}. In the sequel, some definitions and lemmas from the standard linear system theory~\cite{chen1995linear} are firstly reviewed.
\begin{definition}
A matrix $A\in {\mathbb R}^{n\times n}$ is said to be Schur when every eigenvalue of $A$ has a magnitude strictly less than one.
\end{definition}
\begin{lemma}[{\cite[page~174]{chen1995linear}}]\label{thm:stability2}
The system $x(k+1) = Ax(k)$ is asymptotically stable if and only if $A$ is Schur.
\end{lemma}

The Laypunov method given below provides a way to check the stability of $x(k+1) = Ax(k)$ and the Schur property of $A$ without actually calculating the eigenvalues.
\begin{lemma}[{\cite[Thm.~5.D5]{chen1995linear}}]\label{thm:Lyapunov-theorem2}
A matrix $A \in {\mathbb R}^{n\times n}$ is Schur if and only if there exists a symmetric matrix $P\in {\mathbb S}^n$ such that
\[
P \succ 0,\quad A^T P A - P \prec 0,
\]
where $P \in {\mathbb S}^n$ is called the Lyapunov matrix.
\end{lemma}

As in the continuous-time case, using~\cref{thm:stability2}, the corresponding state-feedback stabilization problem is stated below.
\begin{problem}[State-feedback stabilization problem]\label{problem:stabilization2}
Find a feedback gain $F \in {\mathbb R}^{m\times n}$ such that $A+BF$ is Schur.
\end{problem}

Now, based on~\cref{thm:Lyapunov-theorem2}, the problem can be formulated as the Lyapunov inequality
\[
P \succeq \varepsilon I,\quad (A + BF)^T P(A + BF)- P \preceq -\varepsilon I,
\]
where $P \in {\mathbb S}^n$ is called the Lyapunov matrix, and the second inequality is in general non-convex.
Moreover, by duality of LTI systems again, it can be equivalently written as
\[
P \succeq \varepsilon I,\quad (A + BF) P(A + BF)^T- P \preceq -\varepsilon I.
\]
\cref{problem:stabilization2} can be converted to the equivalent optimization
\begin{align}
&\inf_{P \in {\mathbb S}^n ,F \in {\mathbb R}^{m \times n} } 0 \label{eq:11}\\
&{\rm s.t.}\quad (A + BF) P(A + BF)^T - P + \varepsilon I = \Phi_1(x) \preceq 0, \nonumber\\
& \varepsilon I-P= \Phi_2(x)\preceq 0 \nonumber
\end{align}
where
\begin{align*}
x = \left[ {\begin{array}{*{20}c}
   P  \\
   F  \\
\end{array}} \right].
\end{align*}

As in the continuous-time case, the above results are well known in the control literature, and more comprehensive introductions of LMI-based control design processes can be found in~\cite{Boyd1994,el2000advances}.
Next, let us consider the mapping
\begin{align*}
x = \left[ {\begin{array}{*{20}c}
   P  \\
   F  \\
\end{array}} \right] \mapsto h\left( {\left[ {\begin{array}{*{20}c}
   P  \\
   F  \\
\end{array}} \right]} \right) = \left[ {\begin{array}{*{20}c}
   P  \\
   {FP}  \\
\end{array}} \right]
\end{align*}
and the corresponding change of variables
\begin{align*}
v = \left[ {\begin{array}{*{20}c}
   P  \\
   M  \\
\end{array}} \right] = h\left( {\left[ {\begin{array}{*{20}c}
   P  \\
   F  \\
\end{array}} \right]} \right) = \left[ {\begin{array}{*{20}c}
   P  \\
   {FP}  \\
\end{array}} \right].
\end{align*}

Using this transformation, a convexification is
\begin{align}
&\inf_{P \in {\mathbb S}^n ,M \in {\mathbb R}^{m \times n} } 0 \label{eq:14}\\
&{\rm s.t.}\quad (AP + BM)P^{ - 1} (AP + BM)^T  - P + \varepsilon I = \Phi _1' (v)\preceq 0,\nonumber\\
& \varepsilon I-P= \Phi_2'(v)\preceq 0. \nonumber
\end{align}
We note that as in the continuous-time case, the above conversion from non-convex to convex problems via change of variables is very well-known and popular technique in the control literature, e.g.,~\cite{de1999new,kim2005new,ghaffari2022robust,de2022new}.

The domain of $\Phi _1'$ is $\{ (P,M) \in {\mathbb S}^n  \times {\mathbb R}^{m \times n} :P \succ 0\}$.
Over the domain, we can prove that $\Phi _1' (v) \preceq 0$ is convex. To prove it, one can check the convexity of the interior
\begin{align*}
{\bf int}({\mathcal F}):=&\{ (P,M) \in {\mathbb S}^n  \times {\mathbb R}^{m \times n} :P \succ \varepsilon I,\\
&(AP + BM)P^{ - 1} (AP + BM)^T  - P + \varepsilon I \prec 0\},
\end{align*}
where ${\bf int}$ denotes the interior. Note that the set ${\bf int}({\mathcal F})$ is convex because after taking the Schur complement~\cite{Boyd1994}, it is equivalently expressed as
\begin{align*}
{\bf int}({\mathcal F}):=&\{ (P,M) \in {\mathbb S}^n  \times {\mathbb R}^{m \times n} :P \succ \varepsilon I,\\
&\left. {\left[ {\begin{array}{*{20}c}
   { - P + \varepsilon I} & {(AP + BM)}  \\
   {(AP + BM)^T } & { - P}  \\
\end{array}} \right] \prec 0} \right\},
\end{align*}
which is convex. In conclusion,~\eqref{eq:11} is initially a nonconvex bilinear matrix inequality problem, but can be convexified using some transformations and manipulations. Moreover, $P$ is nonsingular in ${\mathcal F}$ and ${\mathcal F}'$. Therefore, $h$ is a bijection, and the inverse mapping is given by
\[
h^{ - 1} \left( {\left[ {\begin{array}{*{20}c}
   P  \\
   M  \\
\end{array}} \right]} \right) = \left[ {\begin{array}{*{20}c}
   P  \\
   {MP^{ - 1} }  \\
\end{array}} \right].
\]
By~\cref{def:lossless-convexification},~\eqref{eq:14} is a lossless convexification of~\eqref{eq:11}.
Lastly, we can prove that the original problem in~\eqref{eq:10} satisfies the Slater's condition under a mild assumption.
\begin{claim}\label{claim:5}
Suppose that $(A,B)$ is stabilizable. Then, there exists a sufficiently small $\varepsilon>0$, $P \in {\mathbb S}^n$ and $F\in {\mathbb R}^{m\times n}$ such that
\[
(A+BF)P(A+BF)^T - P \prec -\varepsilon I,\quad \varepsilon I - P \prec 0.
\]
\end{claim}
\begin{proof}
It is clear from Lyapunov theory in~\cref{thm:Lyapunov-theorem2} that if $(A,B)$ is stabilizable, then there exists $P\succ 0$ and $F\in {\mathbb R}^{m\times n}$ such that $(A+BF)P(A+BF)^T - P \prec 0$. Therefore, the proof is completed.
\end{proof}
Therefore,~\eqref{eq:11} admits a strictly feasible solution, and satisfies the Slater's condition.
By~\cref{thm:strong-duality2-matrix}, the problem in~\eqref{problem:stabilization} satisfies strong duality.

Finally, the static output-feedback stabilization problem for discrete-time systems is omitted here for brevity, but this problem can be addressed in similar ways as in the continuous-time cases.

\section*{Conclusion}
In this paper, we have studied strong duality of non-convex semidefinite programming problems (SDPs). It turns out that a class of non-convex SDPs with special structures satisfies strong duality under the Slater's condition. Examples have been given to illustrate the proposed results. We expect that the proposed analysis can potentially deepen our understanding of non-convex SDPs arising in control communities, and promote their analysis based on KKT conditions. In particular, the developed results can be used to reveal connections between several control-related results and SDP dualities as in~\cite{balakrishnan2003semidefinite}. Moreover, the results can be also applied to develop new algorithms for control designs, such as the static output-feedback design~\cite{el1997cone,crusius1999sufficient,geromel1998static}. Another potential topic is to investigate strong duality of non-convex SDPs which can be convexified with several conversions of the problems using Schur complement~\cite{Boyd1994} and its variations. These agendas can be potential future directions.

\bibliographystyle{IEEEtran}
\bibliography{reference}

\begin{thebibliography}{10}
\providecommand{\url}[1]{#1}
\csname url@samestyle\endcsname
\providecommand{\newblock}{\relax}
\providecommand{\bibinfo}[2]{#2}
\providecommand{\BIBentrySTDinterwordspacing}{\spaceskip=0pt\relax}
\providecommand{\BIBentryALTinterwordstretchfactor}{4}
\providecommand{\BIBentryALTinterwordspacing}{\spaceskip=\fontdimen2\font plus
\BIBentryALTinterwordstretchfactor\fontdimen3\font minus
  \fontdimen4\font\relax}
\providecommand{\BIBforeignlanguage}[2]{{%
\expandafter\ifx\csname l@#1\endcsname\relax
\typeout{** WARNING: IEEEtran.bst: No hyphenation pattern has been}%
\typeout{** loaded for the language `#1'. Using the pattern for}%
\typeout{** the default language instead.}%
\else
\language=\csname l@#1\endcsname
\fi
#2}}
\providecommand{\BIBdecl}{\relax}
\BIBdecl

\bibitem{Boyd2004}
S.~Boyd and L.~Vandenberghe, \emph{Convex Optimization}.\hskip 1em plus 0.5em
  minus 0.4em\relax Cambridge University Press, 2004.

\bibitem{vandenberghe1996semidefinite}
L.~Vandenberghe and S.~Boyd, ``Semidefinite programming,'' \emph{SIAM review},
  vol.~38, no.~1, pp. 49--95, 1996.

\bibitem{Boyd1994}
S.~Boyd, L.~El~Ghaoui, E.~Feron, and V.~Balakrishnan, \emph{Linear Matrix
  Inequalities in Systems and Control Theory}.\hskip 1em plus 0.5em minus
  0.4em\relax Philadelphia, PA: SIAM, 1994.

\bibitem{de1999new}
M.~C. De~Oliveira, J.~Bernussou, and J.~C. Geromel, ``A new discrete-time
  robust stability condition,'' \emph{Systems \& control letters}, vol.~37,
  no.~4, pp. 261--265, 1999.

\bibitem{geromel2007h}
J.~C. Geromel, R.~H. Korogui, and J.~Bernussou, ``{$H_2$} and {$H_\infty$}
  robust output feedback control for continuous time polytopic systems,''
  \emph{Control Theory \& Applications, IET}, vol.~1, no.~5, pp. 1541--1549,
  2007.

\bibitem{el2000advances}
L.~El~Ghaoui and S.-l. Niculescu, \emph{Advances in linear matrix inequality
  methods in control}.\hskip 1em plus 0.5em minus 0.4em\relax SIAM, 2000.

\bibitem{geromel1998static}
J.~C. Geromel, C.~De~Souza, and R.~Skelton, ``Static output feedback
  controllers: stability and convexity,'' \emph{IEEE Transactions on Automatic
  Control}, vol.~43, no.~1, pp. 120--125, 1998.

\bibitem{de2002extended}
M.~C. De~Oliveira, J.~C. Geromel, and J.~Bernussou, ``Extended {$H_2$} and
  {$H_\infty$} norm characterizations and controller parametrizations for
  discrete-time systems,'' \emph{International Journal of Control}, vol.~75,
  no.~9, pp. 666--679, 2002.

\bibitem{yao2001stochastic}
D.~D. Yao, S.~Zhang, and X.~Y. Zhou, ``Stochastic linear-quadratic control via
  semidefinite programming,'' \emph{SIAM Journal on Control and Optimization},
  vol.~40, no.~3, pp. 801--823, 2001.

\bibitem{rami2000linear}
M.~A. Rami and X.~Y. Zhou, ``Linear matrix inequalities, {R}iccati equations,
  and indefinite stochastic linear quadratic controls,'' \emph{Automatic
  Control, IEEE Transactions on}, vol.~45, no.~6, pp. 1131--1143, 2000.

\bibitem{henrion2001rank}
D.~Henrion, G.~Meinsma \emph{et~al.}, ``Rank-one {LMI}s and {L}yapunov's
  inequality,'' \emph{IEEE Transactions on Automatic Control}, vol.~46, no.~8,
  pp. 1285--1288, 2001.

\bibitem{lee2019primal}
D.~Lee and J.~Hu, ``Primal-dual {Q}-learning framework for {LQR} design,''
  \emph{IEEE Transactions on Automatic Control}, vol.~64, no.~9, pp.
  3756--3763, 2019.

\bibitem{gattami2010generalized}
A.~Gattami, ``Generalized linear quadratic control,'' \emph{IEEE Transactions
  on Automatic Control}, vol.~55, no.~1, pp. 131--136, 2010.

\bibitem{balakrishnan2003semidefinite}
V.~Balakrishnan and L.~Vandenberghe, ``Semidefinite programming duality and
  linear time-invariant systems,'' \emph{IEEE Transactions on Automatic
  Control}, vol.~48, no.~1, pp. 30--41, 2003.

\bibitem{you2013lagrangian}
S.~You and J.~C. Doyle, ``A {L}agrangian dual approach to the {G}eneralized
  {KYP} lemma,'' in \emph{CDC}, 2013, pp. 2447--2452.

\bibitem{you2015primal}
S.~You, A.~Gattami, and J.~C. Doyle, ``Primal robustness and semidefinite
  cones,'' \emph{arXiv preprint arXiv:1503.07561}, 2015.

\bibitem{bertsekas1999nonlinear}
D.~P. Bertsekas, \emph{Nonlinear programming}.\hskip 1em plus 0.5em minus
  0.4em\relax Athena scientific Belmont, 1999.

\bibitem{farjadnasab2022model}
M.~Farjadnasab and M.~Babazadeh, ``Model-free lqr design by q-function
  learning,'' \emph{Automatica}, vol. 137, p. 110060, 2022.

\bibitem{clarke2022low}
S.~G. Clarke, S.~Byeon, and I.~Hwang, ``A low complexity approach to model-free
  stochastic inverse linear quadratic control,'' \emph{IEEE Access}, vol.~10,
  pp. 9298--9308, 2022.

\bibitem{li2022model}
M.~Li, J.~Qin, W.~X. Zheng, Y.~Wang, and Y.~Kang, ``Model-free design of
  stochastic lqr controller from a primal--dual optimization perspective,''
  \emph{Automatica}, vol. 140, p. 110253, 2022.

\bibitem{esmzad2023maximum}
R.~Esmzad and H.~Modares, ``Maximum-entropy satisficing linear quadratic
  regulator,'' \emph{IEEE Control Systems Letters}, 2023.

\bibitem{sun2021learning}
Y.~Sun and M.~Fazel, ``Learning optimal controllers by policy gradient: Global
  optimality via convex parameterization,'' in \emph{2021 60th IEEE Conference
  on Decision and Control (CDC)}, 2021, pp. 4576--4581.

\bibitem{mohammadi2019global}
H.~Mohammadi, A.~Zare, M.~Soltanolkotabi, and M.~R. Jovanovi{\'c}, ``Global
  exponential convergence of gradient methods over the nonconvex landscape of
  the linear quadratic regulator,'' in \emph{2019 IEEE 58th Conference on
  Decision and Control (CDC)}, 2019, pp. 7474--7479.

\bibitem{fazel2018global}
M.~Fazel, R.~Ge, S.~Kakade, and M.~Mesbahi, ``Global convergence of policy
  gradient methods for the linear quadratic regulator,'' in \emph{International
  conference on machine learning}, 2018, pp. 1467--1476.

\bibitem{bu2019lqr}
J.~Bu, A.~Mesbahi, M.~Fazel, and M.~Mesbahi, ``Lqr through the lens of first
  order methods: Discrete-time case,'' \emph{arXiv preprint arXiv:1907.08921},
  2019.

\bibitem{duan2023optimization}
J.~Duan, W.~Cao, Y.~Zheng, and L.~Zhao, ``On the optimization landscape of
  dynamic output feedback linear quadratic control,'' \emph{IEEE Transactions
  on Automatic Control}, 2023.

\bibitem{duan2022optimization}
J.~Duan, J.~Li, S.~E. Li, and L.~Zhao, ``Optimization landscape of gradient
  descent for discrete-time static output feedback,'' in \emph{2022 American
  Control Conference (ACC)}, 2022, pp. 2932--2937.

\bibitem{khalil2002nonlinear}
H.~K. Khalil, \emph{Nonlinear systems}.\hskip 1em plus 0.5em minus 0.4em\relax
  Upper Saddle River, 2002.

\bibitem{chen1995linear}
C.-T. Chen, \emph{Linear System Theory and Design}.\hskip 1em plus 0.5em minus
  0.4em\relax Oxford University Press, Inc., 1995.

\bibitem{khalil2015nonlinear}
H.~K. Khalil, \emph{Nonlinear control}.\hskip 1em plus 0.5em minus 0.4em\relax
  Pearson, 2015.

\bibitem{kim2005new}
J.~H. Kim, ``New design method on memoryless h-$\infty$ control for singular
  systems with delayed state and control using lmi,'' \emph{Journal of the
  Franklin Institute}, vol. 342, no.~3, pp. 321--327, 2005.

\bibitem{ghaffari2022robust}
V.~Ghaffari, ``A robust predictive observer-based integral control law for
  uncertain lti systems under external disturbance,'' \emph{Journal of the
  Franklin Institute}, vol. 359, no.~13, pp. 6915--6938, 2022.

\bibitem{de2022new}
L.~T. de~Souza, M.~L. Peixoto, and R.~M. Palhares, ``New gain-scheduling
  control conditions for time-varying delayed lpv systems,'' \emph{Journal of
  the Franklin Institute}, vol. 359, no.~2, pp. 719--742, 2022.

\bibitem{fu1997computational}
M.~Fu and Z.-Q. Luo, ``Computational complexity of a problem arising in fixed
  order output feedback design,'' \emph{Systems \& Control Letters}, vol.~30,
  no.~5, pp. 209--215, 1997.

\bibitem{fu2004pole}
M.~Fu, ``Pole placement via static output feedback is {NP}-hard,'' \emph{IEEE
  Transactions on Automatic Control}, vol.~49, no.~5, pp. 855--857, 2004.

\bibitem{blondel1997np}
V.~Blondel and J.~N. Tsitsiklis, ``{NP}-hardness of some linear control design
  problems,'' \emph{SIAM journal on control and optimization}, vol.~35, no.~6,
  pp. 2118--2127, 1997.

\bibitem{el1997cone}
L.~El~Ghaoui, F.~Oustry, and M.~AitRami, ``A cone complementarity linearization
  algorithm for static output-feedback and related problems,'' \emph{IEEE
  Transactions on Automatic Control}, vol.~42, no.~8, pp. 1171--1176, 1997.

\bibitem{crusius1999sufficient}
C.~A. Crusius and A.~Trofino, ``Sufficient {LMI} conditions for output feedback
  control problems,'' \emph{IEEE Transactions on Automatic Control}, vol.~44,
  no.~5, pp. 1053--1057, 1999.

\end{thebibliography}

\end{document}